\documentclass{amsart}
\usepackage[utf8]{inputenc}
\newtheorem{theorem}{Theorem}
\newtheorem{lemma}[theorem]{Lemma}
\newtheorem{proposition}[theorem]{Proposition}

\newtheorem{definition}[theorem]{Definition}
\newtheorem{remark}[theorem]{Remark}
\newtheorem{corollary}[theorem]{Corollary}
\usepackage{hyperref}

\title[Images of polynomials on null-filiform Leibniz algebras]{Polynomial identities and images of polynomials on null-filiform Leibniz algebras}
\author[T. C. de Mello]{Thiago Castilho de Mello}
\address{Universidade Federal de S\~ao Paulo, Instituto de Ci\^encia e Tecnologia}
\email{tcmello@unifesp.br}

\author[M. S. Souza]{Manuela da Silva Souza}\address{Universidade Federal da Bahia, Instituto de Matem\'atica e Estat\'istica}
\email{manuela.souza@ufba.br}

\keywords{Images of polynomials on algebras, Leibniz algebras, null-filiform Leibniz algebras, polynomial identities, L'vov-Kaplansky conjecture}\subjclass[2020]{17A30, 17A32}

\begin{document}

\begin{abstract}
    In this paper we study identities and images of polynomials on null-filiform Leibniz algebras. If $L_n$ is an $n$-dimensional null-filiform Leibniz algebra, we exhibit a finite minimal basis for $\mbox{Id}(L_n)$, the polynomial identities of $L_n$, and we explicitly compute the images of multihomogeneous polynomials on $L_n$. We present necessary and sufficient conditions for the image of a multihomogeneous polynomial $f$ to be a subspace of $L_n$.  For the particular case of multilinear polynomials, we prove that the image is always a vector space, showing that the analogue of the L'vov-Kaplansky conjecture holds for $L_n$. We also prove similar results for an analog of null-filiform Leibniz  algebras in the infinite-dimensional case. 

\end{abstract}

\maketitle

\section{Introduction}
    
    Let $K$ be a field and $X=\{x_1, x_2, \dots\}$ be a countable infinite set. By $K\{X\}$ we denote the free (nonassociative) $K$-algebra freely generated by $X$. We refer to the elements  of $K\{X\}$ as polynomials. If $A$ is an algebra and $f=f(x_1, \dots, x_m)$ is a polynomial, it defines a map (also denoted by $f$)
\[\begin{array}{cccc}
f: & A^m & \longrightarrow & A\\
  & (a_1, \dots, a_m) & \longmapsto & f(a_1, \dots, a_m)
\end{array}\] by evaluation of variables of $f$ on elements of $A$. One may ask what is the image of a given polynomial, or which subsets of $A$ are the image of some polynomial in $K\{ X \}$.

In recent years, images of polynomials have been studied by several authors in different contexts, mostly motivated by the famous open problem known as \emph{L’vov-Kaplansky Conjecture}. Such conjecture asserts that  the image of a multilinear polynomial on the algebra of $n\times n$ matrices is a vector space. In this case, it must be one of the following subspaces: the full matrix algebra, the set of traceless matrices, the set of scalar matrices, or the set $\{0\}$. 

The polynomials whose image is $\{0\}$ are called \emph{polynomial identities} of $A$ and the ones whose image is the set of scalar matrices are called \emph{central polynomials}. The theory of polynomial identities in algebras is a well developed theory and is an important tool in studying images of polynomials. We refer the reader to \cite{Drensky, 4Russos, Bahturin} for good references on polynomial identities in associative and non-associative algebras.

The L'vov-Kaplansky conjecture has a solution  only for $n=2$ \cite{K-BMR} or for $m=2$ (which is a consequence of an old result of Shoda \cite{Shoda} and Albert and Muckenhoupt \cite{AM}, which asserts that any trace zero matrix is a commutator. A simpler proof of such result can be found in \cite{Trace0}).

One can ask if the above conjecture is true for an arbitrary algebra, but it is not difficult to find examples where it fails. For example, the image of the multilinear polynomial $f(x,y)=[x,y]:=xy-yx$ (commutator) on the Grassmann algebra of an infinite-dimensional vector space is not a subspace. So it is interesting to delimit the scope of such a generalized conjecture.
Recent papers have shown that the conjecture is true for several examples of simple algebras: the algebra of real quaternions \cite{MalevQ}, the algebra of real octonions \cite{MalevO}, the Jordan algebra of a non-degenerated bilinear form \cite{MalevJ}. Also, in \cite{VitasSurjective} the author proves that if $A$ is an associative unital algebra with a surjective inner derivation, then the image of $f$ on $A$ is $A$. In particular the image of a multilinear polynomial evaluated on the (simple) algebra of endomorphisms of an infinite-dimensional vector space over $K$ and on the (simple) $n$-th Weyl  algebras is a vector subspace.  For Lie algebras, a solution is known for the Lie algebras $\mathfrak{su}(n)$ and $\mathfrak{so}(n)$ for polynomials of degree up to 4 \cite{Anzis}. 
We refer the reader to \cite{survey} for a comprehensive survey of this line of investigation.

Although the L'vov-Kaplansky conjecture is not true for algebras in general, there are known positive results for non-simple algebras. An important example is the class of algebras of upper triangular matrices over a sufficiently large field \cite{GargatedeMello, Wang_nxn}. Also the same is true for the algebra of $n\times n$ strictly upper triangular matrices (and also for its powers) \cite{Fagundes}. It is worth mentioning that the later algebra is nilpotent. In \cite{LieNilpotent} the authors describe the image of multilinear Lie polynomials of degree 2 evaluated
on a finite-dimensional nilpotent Lie algebra $L$  with $\dim L' \leq 4$.

In this paper, we address the problem of describing the image of a multihomogeneous polynomial evaluated on a special type of nilpotent Leibniz algebras. 

Leibniz algebras are a ``non-commutative" analogue of Lie algebras.
Such algebras have been intensively studied in the last three decades. Mainly, researchers aimed in extending results of Lie algebras to Leibniz algebras, and it has been shown that many important results on Lie algebras have an analogue on Leibniz algebras, for instance an analogue of Engel's Theorem (see \cite{Omirov_Engel}) and a weaker version of Levi's Theorem (see \cite{Barnes_Levi}) hold for  Leibniz algebras. 

The Leibniz algebras we are interested in this papers are the so called \emph{null-filiform} Leibniz algebra. An $n$-dimensional Leibniz algebra $L$ is called null-filiform if for any $1\leq i \leq n+1$, $\dim L^i = (n+1)-i$. In particular, such algebras have maximal nilpotent index.

In this paper we completely solve the problem of describing the image of an arbitrary multihomogeneous polynomial evaluated on an $n$-dimensional null-filiform Leibniz algebra. We present necessary and sufficient conditions for the image of a multihomogeneous polynomial $f$ to be a subspace of $L_n$.  
In particular, we show that the image of a multilinear polynomial is always a vector subspace. 

Also, we present a finite basis for the polynomial identities of an $n$-dimensional null-filiform Leibniz algebra, and a linear basis for its relatively free algebras.

In the last section we consider an infinite-dimensional analog of the null-filiform Leibniz algebras, presenting a finite basis for its identities and describing the image of multilinear and multihomogeneous polynomials on it.

\section{Preliminaries}
In this paper $K$ will be a field.
Any algebra will be over $K$. A Leibniz algebra $L$ is an algebra that satisfies the relation
\begin{align}
\label{Leibidentity}
    x(yz)=(xy)z-(xz)y 
\end{align}
for all $x$, $y$ and $z \in L$. In other words, right multiplication by any element $z$ is a derivation. The above identity is called Leibniz identity. If in addition $x^2 = 0$ for any $x \in L$, then the Leibniz algebra is a Lie algebra. Conversely any Lie algebra is a Leibniz algebra. In this sense, Leibniz algebras can be seen as a non-commutative generalization of Lie algebras. For a given Leibniz algebra $L$, the chain of ideals defined inductively by
\[ L^{1}=L, ~ L^{k+1}=L^{k}L , ~ k\geq 1,\]
is called the lower central series of $L$.
A Leibniz algebra $L$ is nilpotent, if $L^{m}=0$ for some integer $m > 0$. The minimal number $m$ is called the nilpotency index of $L$.
\begin{definition}[\cite{Omirov_Null}]
An $n$-dimensional Leibniz algebra $L$ is called null-filiform if 
$$\dim L^{i}=n+1 -i \mbox{ where } 1 \leq i \leq n+1.$$ 
\end{definition}
It is easy to see that every null-filiform Leibniz algebras has maximal nilpotency index.

The next is an important result concerning null-filiform Leibniz algebras. It was proved in \cite{Omirov_Null} and it states that up to isomorphism there is only one such algebra in each dimension.

\begin{theorem}\label{basis}
An arbitrary $n$-dimensional null-filiform Leibniz algebra is isomorphic to the algebra $L_n = \mbox{span}\{e_{1}, \dots, e_{n}\}$ with multiplication table given by
\[ e_{i}e_{1}= e_{i+1},\, 1 \leq i \leq n-1,\] with all other products defined as zero.
\end{theorem}

In particular, when studying images of polynomials on null-filiform Leibniz algebras, it is enough to study images of polynomials on the algebra with the above multiplication table. In fact, if $\varphi: A \longrightarrow B$ is an isomorphism of algebras, and $f(x_1, \dots, x_m)\in K\{X\}$, then for any $a_1, \dots, a_m\in A$, one has 
\[\varphi(f(a_1, \dots, a_m))=f(\varphi(a_1), \dots, \varphi(a_m)), \] 
as a consequence, $S\subseteq A$ is the image of a polynomial $f$ on $A$ if and only if $\varphi(S)$ is the image of $f$ on $B$.

\section{Identities of null-filiform Leibniz algebras}

Let $X = \{x_1, x_2, x_3, \ldots \}$ be an infinite enumerable set of variables and $\mathcal{D}(X)$ be the free Leibniz algebra of countable rank on the set $X$, i.e., the algebra generated by $X$ which satisfies the following universal property: given a Leibniz algebra $L$, any map $X\rightarrow L$ there exists a unique homomorphism of algebras $\mathcal{D}(X) \rightarrow L$ extending the given map. The elements of $\mathcal{D}(X)$ will be called polynomials and we will say that a polynomial $f = f(x_1, x_2, \ldots, x_k) \in \mathcal{D}(X)$ is a \emph{polynomial identity} of a Leibniz algebra $L$ if $f(a_1, a_2, \ldots, a_k) = 0$ for all $a_1$, $a_2$, $\ldots$, $a_n \in L$, that is, if $f$ lies in the kernel of any homomorphism from $\mathcal{D}(X)$ to $L$. 
The subset of $\mathcal{D}(X)$ of all polynomial identities of $L$ is denoted by $\mbox{Id}(L)$. The set $\mbox{Id}(L)$ is a $T$-ideal of $\mathcal{D}(X)$, that is, an ideal invariant under all endomorphisms of $\mathcal{D}(X)$. We denote by 
$\langle S \rangle^T$ the $T$-ideal generated by $S \subseteq \mathcal{D}(X)$, i.e., the intersection of the $T$-ideals of $\mathcal{D}(X)$ that contain $S$. A set of identities $S \subseteq \mbox{Id}(L)$ is a basis of $\mbox{Id}(L)$, if $\mbox{Id}(L)= \langle S \rangle^T$.
A basis of $Id(L)$ is said to be minimal if does not contain properly a basis of $Id(L)$. In order to simplify the notation we define
$$a b ^{(n)} := (\cdots ((a \underbrace{b)b) \cdots) b}_{n-times}$$ for all $a$, $b$ in a Leibniz algebra $L$. 

\begin{remark}\label{nobrackets}
    By Leibniz identity (\ref{Leibidentity}), any polynomial in $\mathcal{D}(X)$ can be written as a linear combination of left-normed monomials, i.e., monomials of the type
    \[(\cdots((x_{i_1}x_{i_2})x_{i_3}) \cdots x_{i_{m-1}})x_{i_m}.\]
    For this reason, we will omit the brackets in left-normed monomials. For instance, the above monomial will be denoted simply by $x_{i_1}x_{i_2}\cdots x_{i_m}$.
\end{remark}

We start with some basic identities of the $n$-dimensional null-filiform Leibniz algebra, which will be useful later in this paper.

\begin{lemma} Let $L_n$ be a $n$-dimensional null-filiform Leibniz algebra. The polynomials

\begin{align}
			&x_1(x_2x_3)\label{identidadeqq1}\\   &x_1x_2 x_3\cdots x_{n} - x_2 x_1 x_3\cdots x_{n}\label{identidadeqq3}
\end{align}
are polynomial identities of $L_n$.
\end{lemma}

\begin{proof} The verification that the above polynomials are identities for $L_n$ is immediate, since by Theorem \ref{basis} we may assume $L_n = \mbox{span}\{e_{1}, \dots, e_{n}\}$ with
$e_{i}e_{1}= e_{i+1},\, 1 \leq i \leq n-1$ and all the other products are zero.     
\end{proof}

The next  proposition gives some consequences of the above identities.

\begin{proposition}\label{consequences}
Let $I$ be the $T$-ideal of $\mathcal{D}(X)$ generated by \emph{(\ref{identidadeqq1})} and \emph{(\ref{identidadeqq3})}. Then
\begin{enumerate}

\item[(i)] $x_1 x_2\cdots x_{n+1} \label{identidadeqq2} \in I;$
    \item[(ii)] $x_1x_2x_3-x_1x_3x_2 \in I;$
    \item[(iii)] $x_0x_{\sigma(1)}\cdots x_{\sigma(k)} - x_0 x_1 \cdots x_k \in I$, for any $k\geq 1$ and for any $\sigma\in S_k$;
    \item[(iv)] $x_{\sigma(1)}x_{\sigma(2)} \cdots x_{\sigma(n)} - x_1x_2 \cdots x_{n} \in I$,  for any $\sigma \in S_n$.
\end{enumerate}

In particular, if $L_n$ is an $n$-dimensional null-filiform Leibniz algebra, the polynomials in $\emph{(i)}$, $\emph{(ii)}$, $\emph{(iii)}$ and $\emph{(iv)}$ above are identities for $L_n$.
   
\end{proposition}

\begin{proof}
    \begin{enumerate}
        \item[(i)] Replacing $x_1$ by $xy$ in $(\ref{identidadeqq3})$ we obtain $xyx_2x_3\cdots x_n - x_2(xy)x_3 \cdots x_n \in I$. Since $x_2(xy)x_3 \cdots x_n$ is a consequence of (\ref{identidadeqq1}) then $xyx_2x_3\cdots x_n \in I.$
        \item[(ii)] This is a direct consequence of Leibniz identity, once $(x_1x_2)x_3-(x_1x_3)x_2 = x_1(x_2 x_3) \in I.$
        \item[(iii)] One obtains this identity by induction on $k$ using (ii).
        \item[(iv)] This is a direct consequence of (ii) and identity (\ref{identidadeqq3}).
    \end{enumerate}
\end{proof}

\begin{lemma}\label{relativelyfree} Let $S$ be the subset of $\mathcal{D}(X)$ of all monomials 
\begin{enumerate}
\item[(i)] $x_k$ where $k \geq 1$;
 \item[(ii)] $x_jx_{i_1}^{(d_1)}x_{i_2}^{(d_2)} \ldots x_{i_r}^{(d_r)}$ where  $0 < d_1 + \cdots + d_r \leq n - 2$, $1 \leq i_1 < \cdots < i_r$ and $j, r \geq 1$;
 \item[(iii)] $x_{j_1}x_{j_1}^{(d_1)}x_{j_2}^{(d_2)} \ldots x_{j_s}^{(d_s)}$ where  $0 < d_1 + \cdots + d_s =  n - 1$ and $ 1 \leq j_1 < \cdots < j_s$ and $s \geq 1$.
 \end{enumerate}
Then the quotient vector space $\frac{\mathcal{D}(X)}{I}$ is spanned by the set of all elements $g+I$ where $g \in S$.
\end{lemma}

\begin{proof} Let $m(x_{i_1}, x_{i_2}, \ldots, x_{i_r})$ be a monomial of $\mathcal{D}(X) - I$. The proof will be by induction in $\mbox{deg}(m)$. If $\mbox{deg}(m) = 1$ then there exists $j \in \{1, \ldots, r\}$ such that $m(x_{i_1}, x_{i_2}, \ldots, x_{i_r}) = x_{i_j}$. If $\mbox{deg}(m) > 1$, there exists $m_1$, $m_2 \in \mathcal{D}(X)$ such that $m = (m_1)(m_2)$ and $\mbox{deg}(m) > \mbox{deg}(m_1)$, $\mbox{deg}(m) > \mbox{deg}(m_2)$. Since the polynomial $(\ref{identidadeqq1})$ lies in $I$, and $m\not \in I$, we may assume $\deg(m_2)=1$. Then there exists $x_{i_s}$ such that $m \equiv (m_1)x_{i_s}$ (mod $I$). Thus $(m_1)x_{i_s} \equiv (x_{i_t}x_{i_1}^{(d_1)}x_{i_2}^{(d_2)} \ldots x_{i_r}^{(d_r)})x_{i_s}$ (mod $I$), by induction hypothesis. 

Therefore, using {(iii)} of Proposition \ref{consequences}, $m \equiv x_{i_t}x_{i_1}^{(d_1)}x_{i_2}^{(d_2)} \ldots x_{i_s}^{(d_s + 1)} \ldots  x_{i_r}^{(d_r)}$ (mod $I$) with $d_1 + d_2 + \cdots (d_s + 1) + \cdots + d_r \leq n - 1$. If $d_1 + d_2 + \cdots (d_s + 1) + \cdots + d_r \leq n - 2$, we have a monomial of type (ii). If $d_1 + d_2 + \cdots (d_s + 1) + \cdots + d_r =  n - 1$, we may use (iv) of Proposition \ref{consequences} to obtain a monomial of type (iii), modulo $I$.
\end{proof}

The above allow us to determine a finite basis for identities of a null-filiform Leibniz algebra of arbitrary dimension when the field $K$ is infinite.

\begin{theorem}\label{id(L_n)}
Let $K$ be an infinite field and $L_n$ a $n$-dimensional null-filiform Leibniz algebra. Then
\emph{(\ref{identidadeqq1})} and \emph{(\ref{identidadeqq3})}
are a minimal basis of $\mbox{Id}(L_n).$
\end{theorem}

\begin{proof}
    Let $f(x_1, \dots, x_m)\in Id(L_n)$. We are going to show that $f\in I$. Since $K$ is infinite, we may assume $f=f(x_1, \dots, x_m)$ is multihomogeneous of multidegree $(d_1, \dots, d_m)$, with $0\leq d_i\leq n$, for each $i$ ($d_i=0$ means that the variable $x_i$ is missing). By Lemma \ref{relativelyfree}, $f$ is a linear combination of monomials of types (i)-(iii). If $\deg f = 1$, then $f = \alpha x_j$, for some $j\leq n$ and $\alpha\in K$, and $f$ is an identity if and only if $f=0$. If $\deg f = n$, then $f(x_1, \dots, x_m) = \alpha x_1^{(d_1)}\cdots x_m^{(d_m)} \mod I$, for some $\alpha\in K$. Since $f\in Id(L_n)$, $0=f(e_1, \dots, e_1) = \alpha e_n$, we obtain $\alpha = 0$ and $f=0 \mod I$, i.e., $f\in I$. 
    
    Finally, if $1<s=\deg f<n$, then modulo $I$, $f$ can be written as \[f(x_1, \dots, x_m) = \sum_{j = 1}^m \alpha_j x_{j} x_{1}^{(d_1)}x_{2}^{(d_2)} \cdots x_{j}^{(d_j-1)} \cdots x_{m}^{(d_m)},\] with $\alpha_j=0$ whenever $d_j=0$.
       
    For each $j \in \{1, \ldots, m\}$, the substitution $x_{j} = e_{1} + e_2$ and $x_{i} = e_1$ for any $i \neq j$, yields $\alpha_j e_{s+1} + \beta e_{s} = 0$, for some $\beta\in K$ and therefore $\alpha_j = 0$. This proves that $f\in I$ (notice that $s+1\leq n)$. The identity $x_1x_2 x_3\cdots x_{n} - x_2 x_1 x_3\cdots x_{n}$ is not a consequence of $x_1(x_2x_3)$because it is not an identity of the $k$-dimensional null-filiform Leibniz algebra for every $k \geq n+1$.
\end{proof}

The above theorem implies that the monomials of type (i) - (iii) in Lemma \ref{relativelyfree} are linearly independent modulo $Id(L_n)$. In particular, it allows us to compute the dimension of the relatively free algebras of finite rank of $L_n$, $\dfrac{\mathcal{D}(x_1, \dots, x_m)}{Id(L_n)}$, for $m\geq 1$.

\begin{corollary}\label{count} Let $S$ be the set of monomials of type (i) - (iii) in Lemma \ref{relativelyfree}. Then $\{g+Id(L_n)\,|\, g\in S\}$ is a basis for the relatively free algebra of $L_n$, $\dfrac{\mathcal{D}(X)}{Id(L_n)}$. In particular, if $m\geq 1$,
    \[\dim \dfrac{\mathcal{D}(x_1, \dots, x_m)}{Id(L_n)} = 1 + {n+m-1\choose m-1} +  \displaystyle\sum_{s=1}^{n-1}\sum_{l =1}^{min\{m, s\}} l {m\choose l}{s- 1 \choose l - 1} .\]
\end{corollary}

\begin{proof}
    The fact that the above set is a basis of the relatively free algebra follows from the above remark and from Lemma \ref{relativelyfree}. 
    
    By using Lemma \ref{relativelyfree}, we compute the dimension of polynomials of degree $s$ in $\mathcal{D}(x_1, \dots, x_m)/Id(L_n)$, for each $s\in \{0, \dots, n\}$. The element of degree 0 is the unit 1. For a given $1 \leq s<n$ and for each $1 \leq l \leq min\{m, s\}$, there are ${m \choose l}$ ways to choose $l$ variables  among $x_1, \dots, x_m$. The number of positive solutions of the equation  $d_{i_1} + \cdots + d_{i_l} = s$    where $\deg x_{i_j}=d_{i_j} > 0$ is ${s-1\choose s-l}$. For each of these choices, one has $l$ linearly independent monomials of type (ii) in Lemma \ref{relativelyfree}. Finally, for $s=n$, there is only one possible order modulo $I$ for monomials of degree $n$ for each choice of $(d_1, \dots, d_m)$ (monomials of type (iii) in Lemma \ref{relativelyfree}), so we have dimension ${n+m-1\choose m-1}$.
\end{proof}

\section{Images of multilinear polynomials}

In this section we completely describe the image of an arbitrary multilinear polynomial evaluated on an $n$-dimensional null-filiform Leibniz algebra, $L_n$. The results of this section can be obtained as a corollary of the results of the next section, where we deal with multihomogeneous polynomials, nevertheless, we decided to keep this section independently to highlight this important case. Also, although we could state the results of this section for the polynomial written modulo $Id(L_n)$ (as a linear combination of polynomials of types (i) - (iii) of Lemma \ref{relativelyfree}), we chose to write the polynomials as elements of the free Lebniz algebra in the variables $x_1, x_2, \ldots, x_m$, i.e., as linear combination of left-normed monomials, indexed by permutations of $\{1, \dots, m\}$.

In this section we consider $f=f(x_1, \dots, x_m) = \sum_{\sigma\in S_m} \alpha_\sigma x_{\sigma(1)} \cdots x_{\sigma(m)}\in \mathcal{D}(X)$ a multilinear polynomial in the free Leibniz algebra. If we define $\gamma:=\sum_{\sigma\in S_m}\alpha_\sigma$ and for each $j\in \{1, \dots, m\}$, $\gamma_j:=\sum_{\sigma(1)=j} \alpha_{\sigma}$, then using identities of Proposition \ref{consequences}, we obtain

\begin{equation}\label{normalform}
    f(x_1, \dots, x_m) = \left\{\begin{array}{cc} \sum_{j=1}^m \gamma_j x_j x_1 \cdots \widehat{x_j} \cdots x_m  \mod Id(L_n) & \text{ if } m<n \\
\gamma x_1\cdots x_m \mod Id(L_n) & \text{ if } m=n  
\end{array}\right.
\end{equation} where the symbol $\widehat{x_j}$ means the variable $x_j$ is missing.

\begin{lemma}\label{gammajneq0}
    Let $f(x_1, \dots, x_m) = \sum_{\sigma\in S_m} \alpha_\sigma x_{\sigma(1)} \cdots x_{\sigma(m)}\in \mathcal{D}(X)$ be a multilinear polynomial. 
    If for some $j$, $\gamma_j\neq 0$, then $Im(f) \supseteq (L_n)^{m+1}$.
\end{lemma}

\begin{proof}
    If $m \geq n$, the result is trivial since $(L_n)^{m+1}= 0$. So we assume $1 \leq m \leq n-1$. Let $u=\sum_{i=m+1}^n a_ie_i\in (L_n)^{m+1}$. We are going to show that there exists $j\in \{1, \dots, m\}$ and $v\in L$ such that 
    \[f(e_1, \dots, e_1, v, e_1, \dots, e_1) = u,\] where the $v$ above occurs in the $j$-th position.

    First we observe that if we replace $v$ above by $e_i$, for any $2\leq i\leq n-m+1$, we obtain \[f(e_1, \dots, e_1, e_i, e_1, \dots, e_1) = \left(\sum_{\sigma(1)=j}\alpha_\sigma\right)e_{i+m-1}.\]
    Since $\gamma_j= \sum_{\sigma(1)=j}\alpha_\sigma\neq 0$, by defining \[v=\frac{1}{\gamma_j}\sum_{i=1}^{n-m+1}a_{i+m-1}e_i,\] the multilinearity of $f$, implies that
    \[f(e_1, \dots, e_1, v, e_1, \dots, e_1) =\frac{1}{\gamma_j} \sum_{i=1}^na_{i+m-1} f(e_1, \dots, e_1, e_i, e_1, \dots, e_1) = \sum_{i=m}^na_ie_i =u. \]  This proves that $Im(f)\supseteq (L_n)^{m+1}$.
    \end{proof}

\begin{lemma}
    Let $f(x_1, \dots, x_m) = \sum_{\sigma\in S_m} \alpha_\sigma x_{\sigma(1)} \cdots x_{\sigma(m)}\in \mathcal{D}(X)$ be a multilinear polynomial. If $\gamma:=\sum_{\sigma\in S_m} \alpha_\sigma \neq 0$, then $Im(f) = (L_n)^m$. 
\end{lemma}
\begin{proof}
    Of course, if $m>n, Im(f) = \{0\} = (L_n)^{m}$. Hence, we assume $m\leq n$.
    
    We clearly have $Im(f) \subseteq (L_n)^m$. Below we prove the reverse inclusion.
    
    Notice that since $0\neq \gamma = \sum_{j=1}^{m}\gamma_j$, there exists some $j$ such that $\gamma_j\neq 0$. 
    If $j$ is as above, let $w=\sum_{i=m}^na_ie_i\in (L_n)^m$, and let $u=\sum_{i=m+1}^na_ie_i$ if $m < n$, or $u = 0$ if $m = n$. Then $w=a_me_m + u$.

    From the proof of Lemma \ref{gammajneq0}, there exists $v\in L_n$ such that 
    \[f(e_1, \dots, e_1, v, e_1, \dots, e_1) = u,\] where the $v$ above occurs in the $j$-th position.
    Now, the multilinearity of $f$ implies that
    \[f(e_1, \dots, e_1, \frac{a_m}{\gamma}e_1 + v, e_1, \dots, e_1) = w,\] and the proof is complete.
\end{proof}

\begin{lemma}
    Let $f(x_1, \dots, x_m) = \sum_{\sigma\in S_m} \alpha_\sigma x_{\sigma(1)} \cdots x_{\sigma(m)}\in \mathcal{D}(X)$ be a multilinear polynomial. Assume $\gamma=\sum_{\sigma\in S_m} \alpha_\sigma= 0$, then 
    \begin{enumerate}
        \item If for some $j$ $\gamma_j\neq 0$, then $Im(f) = (L_n)^{m+1}$;
        \item If $\gamma_j=0$ for all $j$, then $Im(f) = \{0\}$, that is, $f$ is a polynomial identity
    \end{enumerate}
\end{lemma}

\begin{proof}
    \begin{enumerate}
        \item From Lemma \ref{gammajneq0}, we have $Im(f)\supseteq (L_n)^{m+1}$. Let us prove now the reverse inclusion. Let $v_1, \dots, v_m\in L_n$. For each $k\in \{1, \dots, m\}$, write $v_k = \sum_{l=1}^n a_{kl}e_l$, with $a_{kl}\in K$. Then, opening brackets, we obtain that $f(v_1, \dots, v_m)$ is a linear combination of evaluations of $f$ on basic elements. Notice now that $f(e_1, \dots, e_1) = \left(\sum_{\sigma\in S_m} \alpha_\sigma\right) e_m = \gamma e_m = 0$, and any other such evaluation on basic elements contains at least one basic element different from $e_1$, which implies that each such evaluation will be a multiple of $e_i$, for some $i>m$. As a consequence, $Im(f)\subseteq (L_n)^{m+1}$. 
        
        \item This is an immediate when we write $f$ modulo $Id(L_n)$ as in (\ref{normalform}).

    \end{enumerate}
\end{proof}

By the above Lemmas we obtain that the analogous to the L'vov-Kaplansky conjecture for the null-filiform Leibniz algebras of any dimension is true.

\begin{theorem}\label{multilinear}
    Let $f(x_1, \dots, x_m)$ be a multilinear polynomial in the free algebra of the variety of Leibniz algebras. Then the image of $f$ on $L_n$ is $\{0\}$, $(L_n)^m$ or $(L_n)^{m+1}$. In particular, the image of $f$ is a vector subspace of $L_n$.
\end{theorem}

\section{Images of multihomogeneous polynomials}

In this section we will consider the image of a multihomogeneous polynomial $f(x_1, \dots, x_m)$ over an $n$-dimensional null-filiform Leibniz algebra $L_n$. It turns out that in general the image of such a polynomial is not a vector subspace of $L_n$.

Let $(d_1, \dots, d_m)$ be the multidegree of $f$ and let $d=d_1+\cdots +d_m$ (here we are assuming that each variable occours at least once, i.e., $d_i\geq 1$ for each $i$). By Lemma \ref{relativelyfree}, there exist $\alpha_1, \dots, \alpha_m\in K$ such that, modulo $Id(L_n)$, $f$ may be written as
\begin{equation}\label{fhomogeneous}
f(x_1, \dots, x_m) = \sum_{j=1}^m \alpha_j x_jx_1^{(d_1)} \cdots x_{j-1}^{(d_{j-1})}x_j^{(d_j-1)}x_{j+1}^{(d_{j+1})}\cdots x_m^{(d_m)}.
\end{equation}
In particular, if $d \geq n + 1$ we can choose $\alpha_1, \dots, \alpha_m = 0$ since $f(x_1, \dots, x_m) \in Id(L_n)$, and if $d = n$, we can choose $\alpha_2, \dots, \alpha_m = 0$.

We will consider evaluations of such polynomials by elements of $L_n$. Before that, we need some technical results.

\begin{lemma}\label{product}
    Let $z=\sum_{i=1}^n a_i e_i$ and $w=\sum_{i=1}^n b_ie_i\in L_n$. Then for any $1 \leq s \leq n-1$
    \[zw^{(s)} = b_1^s \sum_{i=1}^{n-s}a_ie_{i+s}.\]
\end{lemma}

\begin{proof} The proof follows from direct computations:
    \[\displaystyle zw^{(s)} = z(b_1e_1)^{(s)} = b_1^s\left(\sum_{i=1}^na_ie_i\right)e_1^{(s)} = b_1^s\sum_{i=1}^{n-s}a_i(e_{i}e_1^{(s)}) = b_1^s\sum_{i=1}^{n-s}a_ie_{i+s}.\]
\end{proof}

Now let $j\in\{1, \dots, m\}$ and let us denote $a_{1,i}$ by $b_i$. 
Under the evaluation $x_k\mapsto \sum_{i=1}^{n} a_{i,k}e_i$, applying the above lemma $m$ times, we obtain
\begin{align*}
    & x_jx_1^{(d_1)} \cdots x_{j-1}^{(d_{j-1})}x_j^{(d_j-1)}x_{j+1}^{(d_{j+1})}\cdots x_m^{(d_m)}\\ = & b_1^{d_1}\cdots b_{j-1}^{d_{j-1}}b_j^{d_j-1}b_{j+1}^{d_{j+1}}\cdots b_m^{d_m}\sum_{i=1}^{n-d+1}a_{i,j}e_{i+d-1}.
\end{align*}
As a consequence, under such evaluation, $f$ becomes
\begin{align}
    & \prod_{l=1}^{m}b_l^{d_l-1}\sum_{j=1}^{m}\alpha_j b_1 \cdots \hat{b_j} \cdots b_m \sum_{i=1}^{n-d+1}a_{i,j}e_{i+d-1} \nonumber \\
    = & \prod_{l=1}^{m}b_l^{d_l} \left(\sum_{j=1}^m \alpha_j\right)e_d + \prod_{l=1}^{m}b_l^{d_l-1}\sum_{j=1}^{m}\alpha_j b_1 \cdots \hat{b_j} \cdots b_m \sum_{i=2}^{n-d+1}a_{i,j}e_{i+d-1} \label{evaluation}
\end{align}

Now we have two cases to consider, depending on whether $\left(\sum_{j=1}^m \alpha_j\right)$ is zero or not.

Let us consider first the case in which the sum of the coefficients of $f$ in (\ref{fhomogeneous}) is zero.

\begin{lemma}\label{sumzero}
    Let $f(x_1, \dots, x_m)$ be a multihomogeneous polynomial of multidegree $(d_1, \dots, d_m)$ such that $f$ is not a polynomial identity for $L_n$. Let $d=d_1+\cdots +d_m$ and write $f$ as in (\ref{fhomogeneous}). If $\sum_{j=1}^m\alpha_j = 0$, then the image of $f$ is $(L_n)^{d+1}$. In particular, the image of $f$ is a vector subspace of $L_n$.    
\end{lemma}

\begin{proof}
     Since $f$ is not a polynomial identity for $L_n$ and $\sum_{j=1}^m\alpha_j = 0$ then $1 \leq d \leq n-1$ and there exists $j$ such that $\alpha_j\neq0$. Without loss of generality, we may assume $\alpha_1\neq0$. From equation (\ref{evaluation}), one easily sees that since $\sum_{j=1} ^m \alpha_j=0$, then the image of $f$ is contained in $(L_n)^{d+1}$. We will show now the reverse inclusion.
    
     Let $u=\sum_{j=d+1}^n\beta_je_j\in (L_n)^{d+1}$. Now consider the substitution $x_i\mapsto e_1$, if $i>1$ and $x_1\mapsto \sum_{j=1}^na_ie_i$.
    Then in equation (\ref{evaluation}), $f$ becomes
    \[a_1^{d_1-1}\alpha_1\sum_{i=2}^{n-d+1}a_ie_{i+d-1} \]
    Now if one sets $a_1:=1$ and $a_i:=\frac{\beta_{i+d-1}}{\alpha_1}$, for $i\geq 2$, we obtain the given element $u$ in $(L_n)^{d+1}$. Hence $Im(f) = (L_n)^{d+1}$. 
\end{proof}

We now consider the case in which the sum of the coefficients of $f$ is nonzero. If $K$ is a field we will denote by $\dot{K}$ the set of nonzero elements of $K$.

\begin{lemma}\label{sumnonzero}
       Let $f(x_1, \dots, x_m)$ be a multihomogeneous polynomial of multidegree $(d_1, \dots, d_m)$ such that $f$ is not a polynomial identity for $L_n$ over an algebraically closed field $K$. Let $d=d_1+\cdots +d_m$ and write $f$ as in (\ref{fhomogeneous}). If $\gamma:=\sum_{j=1}^m\alpha_j \neq 0$ then 
       \[Im f\supseteq \{0\}\cup (\dot{K}\cdot e_d+(L_n)^{d+1}).\]
\end{lemma}

\begin{proof}
    Of course that $0$ is in the image. Since $\sum_{j=1}^m\alpha_j\neq 0$, there exists some $j$ such that $\alpha_j\neq0$. We assume without loss of generality that $\alpha_1\neq 0$.  
    
    Clearly we assume $1 \leq d \leq n$. Let $u=\sum_{j=d}^n\beta_je_j\in (L_n)^{d}$, with $\beta_d\neq0$ and consider again the substitution $x_i\mapsto e_1$, if $i>1$ and $x_1\mapsto \sum_{j=1}^na_ie_i$.
    Then in equation (\ref{evaluation}), $f$ becomes
    \[a_1^{d_1}\left(\sum_{j=1}^m \alpha_j\right)e_d + a_1^{d_1-1}\alpha_1 \sum_{i=2}^{n-d+1}a_ie_{i+d-1}.\]
    By setting $a_1$ to be a $d$-th root of $\frac{\beta_d}{\gamma}$ (here we are using the fact that $K$ is algebraically closed) and $a_i=\frac{\beta_{i+d-1}}{\alpha_1 a_1^{d_1-1}}$, we obtain the given element $u$ in $(L_n)^d$.
\end{proof}

We are now in position to give a complete description of the image of an arbitrary multihomogeneous polynomial on an $n$-dimensional null-filiform Leibniz algebra.


\begin{theorem}\label{homogeneous}
    Let $K$ be an algebraically closed field and $L_n$ a $n$-dimensional null-filiform Leibniz algebra over $K$. Let  $f(x_1, \dots, x_m) \in \mathcal{D}(X)$ be an arbitrary multihomogeneous polynomial of multidegree $(d_1, \dots, d_m)$ and let $d=d_1+\cdots + d_m$. Write $f$ modulo $Id(L_n)$ as \[f(x_1, \dots, x_m) = \sum_{j=1}^m \alpha_j x_jx_1^{(d_1)} \cdots x_{j-1}^{(d_{j-1})}x_j^{(d_j-1)}x_{j+1}^{(d_{j+1})}\cdots x_m^{(d_m)}.\] 
    \begin{enumerate}
    \item If $f$ is a polynomial identity for $L_n$ then the image of $f$ is $\{0\}$.
        \item If  $f$ is not a polynomial identity for $L_n$ and $\sum_{i=1}^m \alpha_i = 0$ then the image of $f$ on $L_n$ is $(L_n)^{d+1}$.
        \item If $f$ is not a polynomial identity for $L_n$ and $\sum_{i=1}^m \alpha_i \neq  0$ then
        \begin{enumerate}
            \item If there exists $j$ such that $d_j=1$ and $\alpha_j\neq0$, then the image of $f$ on $L_n$ is $(L_n)^{d}$;
            \item If $\alpha_j=0$, for all $j$ such that $d_j=1$ then the image of $f$ on $L_n$ is $\{0\} \cup (\dot{K}\cdot e_d +(L_n)^{d+1})$. 
        \end{enumerate}
    \end{enumerate}
\end{theorem}

\begin{proof}
    If $\sum_{j=1}^m\alpha_j=0$, we proved in Lemma \ref{sumzero} that the image of $f$ in $L_n$ is $(L_n)^{d+1}$.
    
    So, assume $\sum_{j=1}^m\alpha_j\neq 0$. By Lemma \ref{sumnonzero} the image of $f$ contains $\{0\}\cup(\dot{K}\cdot e_d+(L_n)^{d+1})$. Assume there exists $j$ such that $d_j=1$ and $\alpha_j\neq0$. Again we may assume $j=1$. 
    Applying the same substitution of variables as in the proof of Lemma \ref{sumnonzero}, and considering that $d_1=1$, we obtain  \[a_1\left(\sum_{j=1}^m \alpha_j\right)e_d +\alpha_1 \sum_{i=2}^{n-d+1}a_ie_{i+d-1}.\]
    Now if $u=\sum_{j=d}^n\beta_je_j\in (L_n)^d$, we set $a_1:=\frac{\beta_{d}}{\sum_{j=1}^m\alpha_j}$ and $a_i=\frac{\beta_i+d-1}{\alpha_1}$, for $i\geq 2$ and we obtain $u$ as an evaluation of $f$ on such elements. This proves (2)(a).
    
    Finally, it remains to prove that in case $2(b)$ if a nonzero element $u = \sum_{j=d}^n\beta_je_j\in (L_n)^d$ is in the image of $f$ then $\beta_d\neq 0$. 
    The most general evaluation of $f$ to consider is the one given in equation (\ref{evaluation}), where we obtain  
    \[\prod_{l=1}^{m}b_l^{d_l} \left(\sum_{j=1}^m \alpha_j\right)e_d + \prod_{l=1}^{m}b_l^{d_l-1}\sum_{j=1}^{m}\alpha_j b_1 \cdots \hat{b_j} \cdots b_m \sum_{i=2}^{n-d+1}a_{i,j}e_{i+d-1}.\]
    Now one can realize that the only way to the coefficient of $e_d$ to be zero is when $b_j=0$ for some $j$. But in this case, since $d_j\geq 2$ if $\alpha_j\neq 0$, we obtain that $\prod_{l=1}^m b_l^{d_l-1}\alpha_j=0$ for all $j$ and the resulting substitution is zero, i.e., any nonzero element $u$ in the image of $f$ has $\beta_d\neq 0$.
\end{proof}

\begin{remark}
    By the last theorem, the image of a multihomogeneous polynomial $f$ written modulo $Id(L_n)$ as in (\ref{fhomogeneous}) is a vector subspace of $L_n$ if and only if $\sum_{j=1}^m\alpha_j=0$ or if $\sum_{j=1}^m\alpha_j\neq 0$ and there exists some $j$ such that $d_j=1$ and $\alpha_j\neq 0$.
\end{remark}

\section{An analog
of null-filiform Leibniz algebras in the infinite-dimensional case}

In this final section we consider an analog
of null-filiform Leibniz algebras in the infinite-dimensional case. For that we consider $L_\infty$ as the algebra spanned by the set $\{e_1, e_2, \dots\}$ with table of  multiplication given by $e_ie_1=e_{i+1}$ for all $i\geq 1$ (with all other products being zero).

Such algebra has been studied before by Omirov in \cite{Omirov_Infinite_Null} and it is a \emph{potentially nilpotent} algebra, i.e., it satisfies $\bigcap_{i=1}^\infty (L_\infty)^{i} = \{0\}$.

\begin{theorem}\label{Id(L_infinite)}
Let $K$ be an infinite field and $L_{\infty}$ be  defined as above. Then
$$x_1(x_2x_3)$$
is a  basis of $\mbox{Id}(L_\infty).$
\end{theorem}

\begin{proof}
    Let $I$ be the $T$-ideal generated by $x_1(x_2x_3)$. By computations similar to Lemma \ref{relativelyfree}, the quotient vector space $\frac{\mathcal{D}(X)}{I}$ is spanned by the set of all elements $ x_jx_{i_1}^{(d_1)}x_{i_2}^{(d_2)} \ldots x_{i_r}^{(d_r)} +I$ where $0 \leq d_1 + \cdots + d_r$, $1 \leq i_1 < \cdots < i_r$ and $j, r \geq 1$ and the proof follows analogously to the proof of Theorem \ref{id(L_n)}.
\end{proof}

In studying polynomial identities in algebras, an important role is played by the so called sequence of codimensions (see for instance \cite{GiambrunoZaicev}, in the case of associative algebras). More specifically if $A$ is an algebra satisfying some polynomial identity, we define its sequence of codimensions  $(c_m(A))_{m \geq 1 }$ as follows:
\[c_m(A)=\dim P_m(A), \ \ \mbox{where} \ \  P_m(A)=\frac{P_m}{P_m \cap Id(A)},\]
and $P_m$ is the vector subspace of $\mathcal{D}(X)$ of multilinear polynomials in the variables $x_1$, $x_2$, $\ldots$, $x_m$.

If the algebra is nilpotent, then the sequence of codimensions is bounded by a constant. In particular, it is not interesting to study the sequence of codimensions of the finite-dimensional null-filiform Leibniz algebras. But the algebra $L_\infty$ is a little be more interesting and we compute its sequence of codimensions below. 

\begin{corollary}
If $K$ is an infinite field, then 
\[c_m(L_\infty) = m\]
for every $m\geq 1$. 
\end{corollary}

\begin{proof} If $K$ is an infinite field, by Theorem \ref{Id(L_infinite)}, a basis for the vector space $P_m(L_\infty)$ is the set of all elements $f+(P_m\cap Id(L_\infty))$ such that
$f=x_1$ if $m=1$ and $f = x_j x_1 \cdots \widehat{x_j} \cdots x_m$ with $1 \leq j \leq m $ if $m \geq 2$.
Therefore $c_m(L_\infty)= m$, for $m\geq 1$.
\end{proof}

In particular, the codimension sequence grows as a linear function. 

Finally, we remark that the arguments in the proofs of Theorem \ref{multilinear} and Theorem \ref{homogeneous} remain valid when considering the algebra $L_\infty$. Then we state the following results.

\begin{theorem}
    Let $f(x_1, \dots, x_m)$ be a multilinear polynomial in the free algebra of the variety of Leibniz algebras. Then the image of $f$ on $L_\infty$ is $\{0\}$, $(L_\infty)^m$ or $(L_\infty)^{m+1}$. In particular, the image of $f$ is a vector subspace of $L_\infty$.
\end{theorem}

\begin{theorem}
    Let $K$ be an algebraically closed field and let $L_\infty$ be a Leibniz algebra over $K$ as above. Let  $f(x_1, \dots, x_m) \in \mathcal{D}(X)$ be an arbitrary multihomogeneous polynomial of multidegree $(d_1, \dots, d_m)$ and let $d=d_1+\cdots + d_m$. Write $f$ modulo $Id(L_\infty)$ as \[f(x_1, \dots, x_m) = \sum_{j=1}^m \alpha_j x_jx_1^{(d_1)} \cdots x_{j-1}^{(d_{j-1})}x_j^{(d_j-1)}x_{j+1}^{(d_{j+1})}\cdots x_m^{(d_m)}.\] 
    \begin{enumerate}
    \item If $f$ is a polynomial identity for $L_\infty$ then the image of $f$ is $\{0\}$.
        \item If  $f$ is not a polynomial identity for $L_\infty$ and $\sum_{i=1}^m \alpha_i = 0$ then the image of $f$ on $L_\infty$ is $(L_\infty)^{d+1}$.
        \item If $f$ is not a polynomial identity for $L_\infty$ and $\sum_{i=1}^m \alpha_i \neq  0$ then
        \begin{enumerate}
            \item If there exists $j$ such that $d_j=1$ and $\alpha_j\neq0$, then the image of $f$ on $L_\infty$ is $(L_\infty)^{d}$;
            \item If $\alpha_j=0$, for all $j$ such that $d_j=1$ then the image of $f$ on $L_\infty$ is $\{0\} \cup (\dot{K}\cdot e_d +(L_\infty)^{d+1})$. 
        \end{enumerate}
    \end{enumerate}
\end{theorem}

\section{Acknowledgements}

The authors would like to thank P.S. Fagundes for useful comments and F.Y. Yasumura for suggesting an improvement in the formula of Corollary \ref{count} and for alerting us that the identity of $L_n$ $x_1\cdots x_{n+1}=0$ is a consequence of polynomials (\ref{identidadeqq1}) and (\ref{identidadeqq3}).

\section{Funding}

This research was supported by São Paulo Research Foundation (FAPESP), grant 2018/23690-6.


\begin{thebibliography}{99}
	
	\bibitem{AM}
	A.~A. Albert and B.~Muckenhoupt.
	\newblock On matrices of trace zeros.
	\newblock {\em Michigan Math. J.}, 4:1--3, 1957.
	
	\bibitem{Anzis}
	B.~E. Anzis, Z.~M. Emrich, and K.~G. Valiveti.
	\newblock On the images of {L}ie polynomials evaluated on {L}ie algebras.
	\newblock {\em Linear Algebra Appl.}, 469:51--75, 2015.
	
	\bibitem{Omirov_Engel}
	S.~A. Ayupov and B.~A. Omirov.
	\newblock On {L}eibniz algebras.
	\newblock In {\em Algebra and operator theory ({T}ashkent, 1997)}, pages 1--12.
	Kluwer Acad. Publ., Dordrecht, 1998.
	
	\bibitem{Omirov_Null}
	S.~A. Ayupov and B.~A. Omirov.
	\newblock On some classes of nilpotent {L}eibniz algebras.
	\newblock {\em Sibirsk. Mat. Zh.}, 42(1):18--29, i, 2001.
	
	\bibitem{Bahturin}
	Y.~A. Bahturin.
	\newblock {\em Identical relations in {L}ie algebras}.
	\newblock VNU Science Press, b.v., Utrecht, 1987.
	\newblock Translated from the Russian by Bahturin.
	
	\bibitem{Barnes_Levi}
	D.~W. Barnes.
	\newblock On {L}evi's theorem for {L}eibniz algebras.
	\newblock {\em Bull. Aust. Math. Soc.}, 86(2):184--185, 2012.
	
	\bibitem{Drensky}
	V.~Drensky.
	\newblock {\em Free algebras and {PI}-algebras}.
	\newblock Springer-Verlag Singapore, Singapore, 2000.
	\newblock Graduate course in algebra.
	
	\bibitem{Fagundes}
	P.~S. Fagundes.
	\newblock The images of multilinear polynomials on strictly upper triangular
	matrices.
	\newblock {\em Linear Algebra Appl.}, 563:287--301, 2019.
	
	\bibitem{GargatedeMello}
	I.~G. Gargate and T.~C. de~Mello.
	\newblock Images of multilinear polynomials on {$n$} {$\times$} {$n$} upper
	triangular matrices over infinite fields.
	\newblock {\em Israel J. Math.}, 252(1):337--354, 2022.
	
	\bibitem{GiambrunoZaicev}
	A.~Giambruno and M.~Zaicev.
	\newblock {\em Polynomial identities and asymptotic methods}, volume 122 of
	{\em Mathematical Surveys and Monographs}.
	\newblock American Mathematical Society, Providence, RI, 2005.
	
	\bibitem{Trace0}
	W.~Kahan.
	\newblock Only commutators have trace zero.
	\newblock {\em unpublished note},
	\href{https://people.eecs.berkeley.edu/~wkahan/MathH110/trace0.pdf}{https://people.eecs.berkeley.edu/~wkahan/MathH110/trace0.pdf},
	1999.
	
	\bibitem{MalevO}
	A.~Kanel-Belov, S.~Malev, C.~Pines, and L.~Rowen.
	\newblock The images of multilinear and semihomogeneous polynomials on the
	algebra of octonions.
	\newblock {\em arXiv:2204.07139}, 2022.
	
	\bibitem{K-BMR}
	A.~Kanel-Belov, S.~Malev, and L.~Rowen.
	\newblock The images of non-commutative polynomials evaluated on {$2\times2$}
	matrices.
	\newblock {\em Proc. Amer. Math. Soc.}, 140(2):465--478, 2012.
	
	\bibitem{survey}
	A.~Kanel-Belov, S.~Malev, L.~Rowen, and R.~Yavich.
	\newblock Evaluations of noncommutative polynomials on algebras: methods and
	problems, and the {L}'vov-{K}aplansky conjecture.
	\newblock {\em SIGMA Symmetry Integrability Geom. Methods Appl.}, 16:Paper No.
	071, 61, 2020.
	
	\bibitem{Wang_nxn}
	Y.~Luo and Y.~Wang.
	\newblock On {F}agundes-{M}ello conjecture.
	\newblock {\em J. Algebra}, 592:118--152, 2022.
	
	\bibitem{MalevQ}
	S.~Malev.
	\newblock The images of noncommutative polynomials evaluated on the quaternion
	algebra.
	\newblock {\em J. Algebra Appl.}, 20(5):Paper No. 2150074, 8, 2021.
	
	\bibitem{MalevJ}
	S.~Malev, R.~Yavich, and R.~Shayer.
	\newblock Evaluations of multilinear polynomials on low rank {J}ordan algebras.
	\newblock {\em Comm. Algebra}, 50(7):2840--2845, 2022.
	
	\bibitem{LieNilpotent}
	N.~Nehra and S.~Rani.
	\newblock Image of lie polynomial of degree 2 evaluated on nilpotent lie
	algebra.
	\newblock {\em Comm. Algebra}, pages 46--62, 2023.
	
	\bibitem{Omirov_Infinite_Null}
	B.~A. Omirov.
	\newblock Thin {L}eibniz algebras.
	\newblock {\em Mat. Zametki}, 80(2):251--261, 2006.
	
	\bibitem{Shoda}
	K.~Shoda.
	\newblock Einige {S}\"{a}tze \"{u}ber {M}atrizen.
	\newblock {\em Jpn. J. Math.}, 13(3):361--365, 1937.
	
	\bibitem{VitasSurjective}
	D.~Vitas.
	\newblock Multilinear polynomials are surjective on algebras with surjective
	inner derivations.
	\newblock {\em J. Algebra}, 565:255--281, 2021.
	
	\bibitem{4Russos}
	K.~A. Zhevlakov, A.~M. Slinko, I.~P. Shestakov, and A.~I. Shirshov.
	\newblock {\em Rings that are nearly associative}, volume 104 of {\em Pure and
		Applied Mathematics}.
	\newblock Academic Press, Inc. [Harcourt Brace Jovanovich, Publishers], New
	York-London, 1982.
	\newblock Translated from the Russian by Harry F. Smith.
	
\end{thebibliography}
\end{document}